\documentclass{amsart}
\usepackage{latexsym,amsmath,amssymb,amscd,amsfonts}
  \usepackage{paralist}
 \usepackage[colorlinks=true]{hyperref}
\hypersetup{urlcolor=blue, citecolor=red}

\newtheorem{theorem}{Theorem}[section]
\newtheorem{corollary}[theorem]{Corollary}

\newtheorem{lemma}[theorem]{Lemma}
\newtheorem{proposition}[theorem]{Proposition}
\theoremstyle{definition}
\newtheorem{definition}[theorem]{Definition}
\newtheorem{remark}[theorem]{Remark}
\newtheorem{example}[theorem]{Example}

\usepackage[arrow, matrix, curve]{xy}
\usepackage{eucal}

\DeclareMathOperator{\pr}{pr}

\newcommand{\TT}{{\mathbb{T}}}
\newcommand{\V}{{\mathcal{V}}}

\newcommand{\tg}{{\mathsf{t}}}
\newcommand{\s}{{\mathsf{s}}}
\newcommand{\m}{{\mathsf{m}}}
\newcommand{\rr}{{\rightrightarrows}}

\newcommand{\R}{\mathbb{R}}

\newcommand{\sfe}{\mathbb{S}}
\newcommand{\inv}{^{-1}}

\newcommand{\N}{\mathbb{N}}

\newcommand{\mx}{\mathfrak{X}}
\newcommand{\lie}[1]{\mathfrak{#1}}
\newcommand{\dr}{\mathbf{d}}

\newcommand{\ip}[1]{{\mathbf{i}}_{#1}}
\newcommand{\an}[1]{\arrowvert_{#1}}

\DeclareMathOperator{\Id}{Id}
\DeclareMathOperator{\rk}{rank}
\DeclareMathOperator{\dom}{Dom}

\DeclareMathOperator{\Exp}{Exp}

\title[The leaf space of a multiplicative foliation]
      {The leaf space of a multiplicative foliation}

\author[M. Jotz]{}


 \email{madeleine.jotz@a3.epfl.ch}

\thanks{The author was supported by Swiss NSF grant 200021-121512.}

\begin{document}
\maketitle

% Enter the first author's name and address:
\centerline{\scshape M. Jotz }
\medskip
{\footnotesize
% please put the address of the first author
 \centerline{Section de Math\'ematiques}
   \centerline{Ecole Polytechnique
  F{\'e}d{\'e}rale de Lausanne}
   \centerline{CH-1015 Lausanne, Switzerland}
}

\bigskip

% The name of the associate editor will be entered by an editorial staff
% "Communicated by the associate editor name" is not needed for special issue.

%The abstract of your paper
\begin{abstract}
We show that if a smooth multiplicative subbundle
$S\subseteq TG$ on a groupoid $G\rr P$ is involutive and  satisfies completeness conditions,
 then its leaf space
$G/S$ inherits a groupoid structure over the space  of leaves of $TP\cap S$ in $P$. 

As an application, a special class of Dirac groupoids is shown to
project by  forward Dirac maps to Poisson groupoids.
\end{abstract}

\begin{center}
\emph{Dedicated to Tudor Ratiu for his 60th birthday.}
\end{center}

\tableofcontents

\section{Introduction}
Let $G$ be a Lie group with Lie algebra $\lie g=T_eG$ and multiplication map
$\mathsf m:G\times G\to G$. 
Then the tangent space $TG$ of $G$ is also a Lie group with unit $0_e\in\lie g$
and  multiplication map $T\mathsf m:TG\times TG\to TG$. 
A multiplicative 
distribution $S\subseteq TG$ is a distribution on $G$, i.e., $S(g):=S\cap T_gG$ is a vector subspace of 
$T_gG$ for all $g\in G$,  that is in addition  a subgroup
of $TG$. Since at each $g\in G$, $S(g)$ is a vector 
subspace of $T_gG$, the zero section of $TG$ is contained in $S$. Thus,
using $T_{(g,h)}\mathsf m(0_g, v_h)=T_hL_gv_h$ for any $g,h\in G$ and
$v_h\in T_hG$, where $L_g:G\to G$ is the left translation by $g$, we find that the 
distribution $S$ is left invariant. It follows that $S$ is a smooth left invariant
 subbundle of $TG$
defined by $S(g)=\lie s^L(g)$, where $\lie s$ the  vector subspace $S(e)=S\cap\lie g$ of $\lie g$. 
In the same manner, $S$ is right invariant 
and we find thus that $\lie s$ is invariant under the adjoint action 
of $G$ on $\lie g$. Hence, $\lie s$ is an ideal in $\lie g$
and the subbundle $S\subseteq TG$ is completely integrable
in the sense of Frobenius. Its leaf $N$ through the unit element $e$ of $G$ is a normal subgroup of $G$
and since the leaf space $G/S$ of $S$ is equal to $G/N$, it inherits 
a group structure from $G$ such that the projection $G\to G/N$ is a homomorphism of groups.
If $N$ is closed in $G$, the foliation by $S$ is simple
and the leaf space $G/S=G/N$ is a Lie group such that the projection is a smooth surjective submersion.
Similar discussions leading to this result can be found in \cite{Ortiz08,Jotz11a}.

\medskip

If $G\rr P$ is a Lie groupoid, its tangent space $TG$  is, in the same manner as in 
the Lie group case, a Lie groupoid over the tangent space $TP$ of the units.
 A distribution $S\subseteq TG$  on $G$ is multiplicative
if it is a subgroupoid of $TG\rr TP$.
The problem studied in this paper is hence natural: we ask how the  simple fact
stated above about the leaf space of a
multiplicative
distribution on a Lie group generalizes to Lie groupoids.
We consider smooth multiplicative and involutive \emph{subbundles}
$S\subseteq TG$ (that is, distributions of constant rank). 
The foliation of $G$ by the leaves of $S$ is then said to be \emph{multiplicative}
and the intersection $TP\cap S$  is  automatically an involutive subbundle of $TP$.
We show that  the leaf space 
$G/S$ inherits the source, target, unit inclusion  and 
inversion maps of a groupoid over the space of leaves in $P$  of $TP
\cap S$, that are all compatible with the projections 
$G\to G/S$ and $P\to P/(TP\cap S)$.  
To be able to define a multiplication on the leaf space,
 we have to assume that a special family of vector fields spanning $S$
 on the Lie groupoid  
is complete in a sense that will be explained, and that a compatibility condition
on the intersections of the leaves with the $\s$-fibers is satisfied.

If the foliations $S$ and $S\cap TP$ are simple, we get hence, under these conditions,
the structure of a  \emph{Lie} groupoid
 on the leaf space $G/S\rr P/(S\cap TP)$, such that the projection is a Lie groupoid morphism.

Using a partial $S\cap TP$-connection on 
$AG/(AG\cap S)$ that is naturally induced by the multiplicative, involutive subbundle 
$S\subseteq TG$ \cite{JoOr11}, we show that 
$\Gamma(AG)$ is spanned as a $C^\infty(P)$-module 
by sections which left invariant images leave $S$ invariant. 
We show that under (rather strong) conditions on these special left invariant vector fields
that leave the foliation invariant,  the completeness condition on $S$ is not necessary and 
the compatibility condition that is necessary for the multiplication is satisfied.

\medskip

We illustrate the theory by four examples. The first two examples 
(Examples \ref{ex-basegp} and \ref{ex-VB}) illustrate the result of the main 
theorem, and the third example
(Example \ref{example_of_henrique}) shows that we recover as a special case
 the quotient of a Lie groupoid by a Lie group action by  groupoid morphisms.
This example suggests that our main theorem (Theorem \ref{groupoid_leaf_space})
could be useful for discrete mechanics on groupoids (see Remark \ref{geom_mech_remark}).
The fourth example (Example \ref{example_of_marco}, which is taken from \cite{Zambon10})
 is an example of a multiplicative foliation on a Lie groupoid that doesn't satisfy
the necessary  condition and which space of leaves fails
hence to inherit a groupoid multiplication.

\medskip

As an application, we show that if $(G\rr P, \mathsf D_G)$ is a
Dirac groupoid such that the characteristic distribution $TG\cap\mathsf D_G$ 
has constant rank and its leaf space is simple and satisfies the 
required conditions, then the quotient Lie groupoid
inherits the structure of a Poisson groupoid such that the quotient 
map is a forward Dirac map and a morphism of Lie groupoids.
 Again, this is a fact that is automatically true
for an integrable Dirac Lie group 
if the leaf $N$ through the unit $e$ of the characteristic foliation is closed 
in the group (\cite{Jotz11a}).

\subsubsection*{\textbf{Notations and conventions}}
Let $M$ be a smooth manifold. We will denote by $\mx(M)$ and $\Omega^1(M)$ the
spaces of (local) smooth sections of the tangent and the cotangent bundle of $M$,
respectively. For an arbitrary vector bundle $\mathsf E\to M$, the space of
(local) sections of $\mathsf E$ will be written $\Gamma(\mathsf E)$. 
We will write $\dom(\sigma)$ for the open subset of the smooth manifold $M$
where the local section $\sigma\in\Gamma(\mathsf E)$ is defined.

\medskip

A distribution $F$ on a smooth manifold $M$
is a subset $F\subseteq TM$ such that $F(m):=F\cap T_mM$ is a vector subspace
of $T_mM$ for all $m\in M$.
The distribution is smooth if for all
$v_m\in F(m)$, there exists a smooth vector field 
$X\in\mx(M)$ with values in $F$ such that $X(m)=v_m$.
Note that in general, the rank of 
$F(m)$ depends on $m$. If not, then $F$ is simply a subbundle of $TM$.

The \emph{pullback} or \emph{restriction} of a distribution $F\subseteq TM$ to an embedded
 submanifold $N$ of $M$  will be written $F\an{N}$. We say that the distribution
$F$ has constant rank on $N$ if the pullback $F\an{N}$ is a vector bundle over $N$, i.e.,
if all its fibers have the same dimension.

\medskip

Assume that $F$ is an involutive \emph{subbundle} of $TM$, hence completely integrable 
in the sense of Frobenius. We say that $F$, or the foliation defined by $F$,
is \emph{simple} if the leaf space $M/F$ of the foliation is 
a smooth manifold such that the projection $M\to M/F$ is a smooth surjective submersion.

\medskip

We write  $\mathsf E\times_M\mathsf F$ for the direct sum 
of two vector bundles $\mathsf E\to M$ and $\mathsf F\to M$.
The \emph{Pontryagin bundle} of $M$ is the direct sum $TM\times_M T^*M\to M$.
\footnote{Note that this vector bundle is also
sometimes called the \emph{generalized tangent bundle} in the literature.}

\medskip

Assume finally that $M,N$ are smooth manifolds and  $f:M\to N$
is a smooth surjective submersion. The tangent space 
to the $f$-fibers will be written $T^fM=\ker(Tf:TM\to TN)\subseteq TM$.

\section{The Pontryagin groupoid of a Lie groupoid}\label{generalities}

The general theory of Lie groupoids and their Lie algebroids can be found 
in \cite{Mackenzie05}, \cite{MoMr03}. We fix here first of all some notations and conventions.

A groupoid $G$ with base $P$ will be written  $G\rr P$. 
The set $P$ will be considered  most of the time as a subset of $G$, that is, the unity
$1_p$ will be identified 
 with $p$ for all $p\in P$.
A \emph{Lie groupoid}
is a groupoid $G$ on base $P$ together with the structures of \emph{smooth Hausdorff manifolds}
on $G$ and $P$ such that
the maps $\s,\tg:G\to P$ are \emph{surjective submersions}, and such that
the object inclusion 
map $\epsilon:P\hookrightarrow G$, $p\mapsto 1_p$  and the partial multiplication
$\mathsf m:G\times_PG\to G$ are all smooth.

Let $g\in G$, then the \emph{right translation by} $g$ is
\[R_g:\s\inv(\tg(g))\to\s\inv(\s(g)), \qquad h \mapsto R_g(h)=h\star g. 
\]
The \emph{left translation} $L_g$ is defined in an analogous manner.
Given a (local) bisection $K$ of $G\rr P$, we write also $R_K$ for the right-translation.

\medskip

In this paper, the Lie algebroid of the Lie groupoid $G\rr P$ 
is $AG:=T^\tg_PG$, equipped with the anchor map
$T\s\an{AG}$ and the Lie bracket defined by the left invariant vector fields.

\medskip

We give in the following  the induced 
Lie groupoid structures on the tangent, cotangent and Pontryagin bundle 
of a Lie groupoid.

\subsubsection*{\textbf{The tangent prolongation of a Lie groupoid}}

Let $G\rr P$ be a Lie groupoid. Applying the tangent functor to each of the
maps defining $G$ yields a Lie groupoid structure on $TG$ with base $TP$,
source $T\s$, target $T\tg$  and multiplication $T\m:T(G\times _PG)\to TG$.
The identity at $v_p\in T_pP$ is $1_{v_p}=T_p\epsilon v_p$.
This defines  the \emph{tangent prolongation $TG\rr TP$ of $G\rr P$}
or the \emph{tangent groupoid associated to $G\rr P$}.

\subsubsection*{\textbf{The cotangent Lie groupoid defined by a Lie groupoid}}
If $G\rr P$ is a Lie groupoid, then there is also an induced 
Lie groupoid structure on $T^*G\rr\, A^*G= (TP)^\circ$. The source map
$\hat\s:T^*G\to A^*G$ is given by 
\[\hat\s(\alpha_g)\in A_{\s(g)}^*G \text{ for } \alpha_g\in T_g^*G,\qquad 
\hat\s(\alpha_g)(u_{\s(g)})=\alpha_g(T_{\s(g)}L_gu_{\s(g)})\]
for all $u_{\s(g)}\in A_{\s(g)}G$, 
and the target map $\hat\tg:T^*G\to A^*G$
is given by \[\hat\tg(\alpha_g)\in A_{\tg(g)}^*G,
\qquad \hat\tg(\alpha_g)(u_{\tg(g)})
=\alpha_g\bigl(T_{\tg(g)}R_g(u_{\tg(g)}-T_{\tg(g)}\s u_{\tg(g)})\bigr)\]
for all $u_{\tg(g)}\in A_{\tg(g)}G$.
If $\hat\s(\alpha_g)=\hat\tg(\alpha_h)$, then the product $\alpha_g\star\alpha_h$
is defined
by 
\[(\alpha_g\star\alpha_h)(v_g\star v_h)=\alpha_g(v_g)+\alpha_h(v_h)
\]
for all composable pairs $(v_g,v_h)\in T_{(g,h)}(G\times_P G)$.
This Lie groupoid structure was introduced in \cite{CoDaWe87}
and is explained for instance in \cite{CoDaWe87}, 
\cite{Pradines88} and \cite{Mackenzie05}.

\subsubsection*{\textbf{The ``Pontryagin groupoid'' of a Lie groupoid}}
If $G\rr P$ is a Lie groupoid, there is hence 
an induced Lie groupoid structure on $\mathsf P_G=TG\times_G T^*G$
over $TP\times_P A^*G$. 
We will write $\TT\tg$ for the target map
\[\begin{array}{ccc}
\TT\tg:TG\times_G T^*G&\to& TP\times_P A^*G\\
(v_g,\alpha_g)&\mapsto&\left(T\tg(v_g),\hat\tg(\alpha_g)\right)
\end{array},
\]
$\TT\s$ for the source map
\[
\TT\s:TG\times_G T^*G\to TP\times_P A^*G
\]
and $\TT\epsilon$, $\TT\mathsf i$, $\TT\mathsf \m$ for the 
embedding of the units, the inversion map and the multiplication of this Lie groupoid.

\section{Multiplicative subbundles of the tangent space $TG$.}
We start with the definition of a multiplicative distribution on a Lie groupoid $G\rr P$. Later, we will
mostly be interested in multiplicative \emph{subbundles} of the tangent bundle $TG$. Yet, as we will see 
in  Examples \ref{ex-basegp}
and \ref{ex-VB}, our main result can already be valid 
 in the more general setting 
of smooth integrable multiplicative distributions.
Since the discussion about the Lie group case in the introduction
also starts from (even non necessarily smooth) distributions,
we prefer to give the general definition here.

\begin{definition}
Let $G\rightrightarrows P$ be a Lie groupoid and $TG\rr TP$ its tangent prolongation.
A distribution $S\subseteq TG$ is \emph{multiplicative} if
$S$ is a (set) subgroupoid of $TG\rr TP$.  
\end{definition}
Note that this means in particular that $T_g\s(v_g)\in S(\s(g))$ and 
$T_g\tg(v_g)\in S(\tg(g))$ for all $g\in G$ and $v_g\in S(g)$.

\subsection{Examples}
The two examples in this subsection illustrate the general theory in the following
sections. We study two completely integrable, multiplicative smooth distributions on special
classes of Lie groupoids 
and show that their  spaces of leaves inherit groupoid structures.
In the general theory, we will have to assume that the distributions have constant rank, 
that the elements of  a special  family of vector fields spanning $S$ are \emph{complete}, and that 
a compatibility condition on the leaves and the multiplication is satisfied.

We will see in both examples that the leaf spaces inherit groupoid structures, although the studied
distributions are here not necessarily \emph{subbundles} of the tangent space. 
This shows that our main theorem can also
hold for leaf spaces of singular distributions.
\begin{example}\label{ex-basegp}  
Let $M$ be a smooth manifold and $M\times M\,\rr\,M$ the pair groupoid associated to it. 
If $F$ is a smooth distribution on $M$, then 
$S:=F\times F\subseteq TM\times TM\simeq T(M\times M)$ is a
smooth multiplicative distribution in $TM\times TM\,\rr \,TM$. Its intersection
with $T\Delta_M$ is $\Delta_{F}$, 
where $\Delta_{F}(m,m)=\{(v_m,v_m)\mid v_m\in F(m)\}$, which is
also a smooth distribution on $\Delta_M$.

If $F$ is completely integrable in the sense of Stefan and Sussmann, 
then $S$ and $\Delta_{F}$ are also completely integrable.
Let $\pr_{S}:M\times M\to  (M\times M)/S$ and $\pr_{F}:M\to
M/F$ be  the quotient maps. If $L_m$, respectively
$L_n$ is the leaf of $F$ through $m$, respectively $n$, then 
the leaf of $S$ through
$(m,n)\in M\times M$ is $L_{(m,n)}=L_m\times L_n$. Thus, the space $(M\times M)/S$
of leaves of $S$  coincides with $M/F\times
M/F$ via the map $\Phi:(M\times M)/S\to M/F\times M/F$, 
$\pr_{S}(m,n)\mapsto (\pr_F(m), \pr_F(n))$.

Since 
the following diagram commutes, 
\begin{displaymath}
\begin{xy}
\xymatrix{
M\times M\ar[d]_{\pr_S}\ar[rd]^{\pr_F\times \pr_F}& \\
(M\times M)/S\ar[r]_\Phi&M/F\times M/F
}
\end{xy}
\end{displaymath}
it is easy to check that $\Phi$ is a homeomorphism.

The groupoid structure on $(M\times M)/S\simeq M/F\times M/F$ is just the pair groupoid structure
$M/F\times M/F\,\rr\, M/F$ and $(\pr_S,\pr_F)$ is a groupoid morphism.
That is, the following diagram commutes
\begin{displaymath}
\begin{xy}
\xymatrix{
M\times M\ar[r]^{\pr_S}\ar@<.6ex>^\s[d]
\ar@<-.6ex>_\tg[d]&(M\times M)/S
\ar@<.6ex>^\s[d]
\ar@<-.6ex>_\tg[d]\\
M\ar[r]_{\pr_F}& M/F
}
\end{xy}
\end{displaymath}
and we have $\pr_S((m,n)\star(n,p))=\pr_S(m,n)\star\pr_S(n,p)$
for all $m,n,p\in M$.

\medskip

Note that in general, $(M\times M)/S\to M/F$ 
is not a \emph{Lie} groupoid since the quotients 
$(M\times M)/S$ and $M/F$ are not necessarily smooth manifolds. 
 If $F$ is an involutive subbundle of $TM$ such that the leaf space
$M/F$ is  a smooth manifold and $\pr_F:M\to M/F$ a
smooth surjective submersion, then the induced groupoid $(M\times M)/S\,\rr\,
M/F$ is a Lie groupoid.
\end{example}

\begin{example}\label{ex-VB}
Let $M$ be a smooth manifold and $\mathsf p:\R^k\times M\to M$ the trivial
vector bundle over $M$. Then $\R^k\times M\,\rr\, M$ has the structure 
of a Lie groupoid over the base $M$.

We identify in the following $T(\R^k\times M)$ with $\R^k\times \R^k\times
TM$, and we write $T_{(x,m)} (\R^k\times M)=\{x\}\times\R^k\times T_mM$.
The source and target maps in $T(\R^k\times M)\,\rr\, TM$ are then
$T\s(x,v,v_m)=T\tg(x,v,v_m)=v_m$ and
the partial multiplication is given by
$(x,v,v_m)\star(y,w,v_m)=(x+y,v+w,v_m)$. Hence, the groupoid $T(\R^k\times
M)\,\rr\, TM$ is the vector bundle groupoid $\R^{2k}\times TM\,\rr \,TM$.

Consider a completely integrable, smooth distribution $F$ on $M$
and a vector subspace $W\subseteq \R^k$. It is easy to see that 
$S:=\R^k\times W\times F$ is an integrable multiplicative
 distribution in $T(\R^k\times M)$.

The leaf of $S$ through $(x,m)$ is the set $(x+W)\times L_m$, 
where $L_m$ is the leaf of $F$ through $m$. The intersection
$TM\cap S$ equals $F$ and we have a (singular) foliation $M/F$.
It is easy to see
 that the quotient $(\R^k\times M)/S\simeq (\R^k/W)\times(M/F)$ inherits the structure of 
 a groupoid over $M/F$
such that the diagram
\begin{displaymath}
\begin{xy}
\xymatrix{
\R^k\times M\ar@<.6ex>[d]^\s\ar@<-.6ex>[d]_\tg\ar[r]^{\pr_S\qquad}
& (\R^k/W)\times(M/F)\ar@<.6ex>[d]^\s\ar@<-.6ex>[d]_\tg\\
M \ar[r]_{\pr_F}& M/F
}
\end{xy}
\end{displaymath}
commutes and $(\pr_S,\pr_F)$ is a morphism of groupoids.
Here also, we get a \emph{Lie} groupoid if and only if 
$F$ is a \emph{simple involutive subbundle} of $TM$, i.e.,
its leaf space is a smooth manifold and the projection is a smooth surjective submersion. 
\end{example}

\begin{remark}
Note that in Example \ref{ex-basegp}, 
the quotient groupoid $(M\times M)/S\,\rr\, M/F$ 
is exactly the quotient 
of $M\times M\,\rr\, M$ by the kernel of $(\pr_S, \pr_F)$, i.e., the
normal subgroupoid \[N:=\ker(\pr_S,\pr_F)=\cup_{m\in M}(L_m\times L_m).\]
In Example \ref{ex-VB},
the groupoid structure on the leaf space does not coincide with  the 
quotient by the kernel $N$ but rather with the quotient 
by the \emph{normal subgroupoid system} $(N, R(\pr_F),\theta)$, 
where  \[N=\ker(\pr_S,\pr_F)=\cup_{m\in M}L_{(0,m)}
=W\times M,\] 
\[R(\pr_F)=\{(m,n)\in M\times M\mid m\sim_F n\}\] are the subgroupoids
of $\R^k\times M\rr M$ and  
of $M\times M\,\rr\,M$, respectively, and $\theta$ is the action of 
$R(\pr_F)$ on $G/N:=(\R^k/W)\times M\to M$ given by
$\theta(m,n)(x+W,n)=(x+W,m)$ for all $x\in\R^k$ and $(m,n)\in R(\pr_F)$.
These two quotients are  explained in more details in \cite{thesis}.

\medskip

This illustrates hence the different
definitions of normal subgroupoids objects in  
\cite{Mackenzie87} and \cite{Mackenzie05}, and the fact that the notion of
 \emph{normal subgroupoid systems of groupoids}
is needed to generalize to groupoids the relation 
between the kernel of surjective homomorphisms and normal subgroups of a group
(see
\cite{Mackenzie05} for more details). 
\end{remark}

\subsection{Properties of multiplicative subbundles of $TG$}
In the following, the subbundle $S\subseteq TG$ is \emph{always assumed 
to be smooth}, even if this is not stated explicitely.
\begin{lemma}\label{constant_rank}
Let $G\rr P$ be a Lie groupoid and
 $S\subseteq TG$ a multiplicative \emph{subbundle}. 
 Then the intersection $S\cap TP$ has constant rank on $P$.
Since this is the set of units of $S$ viewed as a
subgroupoid of $TG$, 
the pair $S\rr (S\cap TP)$
is a Lie groupoid. 

The bundle $S\an{P}$
splitts as $S\an{P}=(S\cap TP)\oplus(S\cap AG)$.
Furthermore, if we denote by $S^\tg$ the intersection 
 $S\cap T^\tg G$ of vector bundles over $G$, 
we have $S^\tg(g)=0_g\star S^\tg(\s(g))
=T_{\s(g)}L_g\left(S^\tg(\s(g))\right)$ for all
$g\in G$. In the same manner, 
$S^\s(g)=S^\s(\tg(g))\star 0_g$ for all
$g\in G$.

As a consequence, the intersections $S\cap T^\tg G$ and 
$S\cap T^\s G$ have  constant rank on $G$.
\end{lemma}

\begin{proof}
We start by showing that the intersection  $S\an{P}\cap TP$
is smooth. If $p\in P$ and $v_p\in S(p)\cap T_pP$, then 
we find a smooth section $X$ of $S$ defined at $p$ such that
$X(p)=v_p$. The restriction of $X$ to $\dom(X)\cap P$ is then a 
smooth section of $S\an{P}$, and, since $\s$ is a smooth surjective submersion,
and $S$ is a subgroupoid of $TG$,
the image $T\s(X\an{S})$ is a smooth section of $S\an{P}\cap TP$. Futhermore, 
we have $T\s(X(p))=T\s(v_p)=v_p$ since $v_p\in T_pP$.

Since the intersection $S\cap TP$
 is a  smooth intersection of vector bundles 
over $P$, we know (for instance by Proposition 4.4 in \cite{JoRaSn11}) that 
it is a vector bundle  on $P$. In particular, it is the set of units of $S$ 
viewed as a subgroupoid of $TG$.

Since $S$ is a vector bundle on $G$, it has constant rank on $G$ and in particular 
on $P$. For each $p\in P$, we can write $S(p)=(S(p)\cap T_pP)\times(S(p)\cap
T^\tg_p G)=(S(p)\cap T_pP)\times(S(p)\cap T^\s_p G)$.
Indeed, if $v_p\in S(p)$, then we can write $v_p=T_p\tg v_p+(v_p-T_p\tg v_p)
=T_p\s v_p+(v_p-T_p\s v_p)$.

From this follows the fact that $S\an{P}\cap T^\tg_P G$ and 
$S\an{P}\cap T^{\s}_PG$ have constant rank on $P$.
If $v_g\in S^\tg(g)$, then we have $T\s(0_{g\inv})=0_{\tg(g)}=T\tg(v_g)$, 
and since $S(g\inv)$ is a vector subspace of $T_{g\inv}G$, we have 
$0_{g\inv}\in S(g\inv)$. Since $S$ is multiplicative, we find hence
$0_{g\inv}\star v_g\in S(\s(g))$. The equality 
$0_{g\inv}\star v_g=T_gL_{g\inv}v_g$ is easy to check. We show the other 
inclusion and the
equality  $S^\s(g)=T_{\tg(g)}R_g\left(S^\s(\tg(g))\right)$
in the same manner.

Since $S\an{P}\cap AG$ and $S\an{P}\cap T_P^\s G$ have constant rank on $P$, we get
from this that $S\cap T^\tg G$ and $S\cap T^\s G$ have constant rank on $G$.
\end{proof}

\begin{corollary}\label{surjectivity}
Let $G\rr P$ be a  Lie groupoid 
and $S$ a multiplicative \emph{subbundle} of $TG$.
 The induced
maps $T_g\s:S(g)\to S(\s(g))\cap T_{\s(g)}P
$ and $T_g\tg:S(g)\to S(\tg(g))\cap T_{\tg(g)}P$
are surjective for each $g\in G$.
\end{corollary}

\begin{proof}
The map $T\s:S/(S\cap T^\s G)\to S\cap TP$
is a well-defined  injective  
vector bundle homomorphism over $\s:G\to P$. 
Since 
\begin{align*}
\rk(S/(S\cap T^\s G))&=\dim((S/(S\cap T^\s G))(g))=\dim(S(g))-\dim(S(g)\cap T_g^\s G)\\
&=\dim(S(\tg(g)))-\dim(S(\tg(g))\cap T_{\tg(g)}^\s G)\\
&=\dim(S(\tg(g))\cap T_{\tg(g)}P)=\rk(S\cap TP)
\end{align*}
for any $g\in G$,
both vector bundles have the same rank, 
and the map is an isomorphism in every fiber. Thus, the claim follows.
\end{proof}

The following corollary (see also a result
in \cite{Mackenzie00} about \emph{star-sections}) will be used often in the following.
\begin{corollary}\label{trelated}
Let $G\rr P$ be a Lie groupoid and $S\subseteq TG$ a multiplicative subbundle.
Let $\bar X$ be a section of $S\cap TP$ defined
on $\dom(\bar X)=:\bar U\subseteq P$. Then there exist
sections $X,Y \in\Gamma(S)$ defined on $U:=\s\inv(\bar U)$,
respectively $V=\tg\inv(\bar U)$ 
such that 
$X\sim_\s \bar X$ and $Y\sim_\tg \bar X$. 
\end{corollary}

\begin{proof}
Since the induced  map $T\s:S/(S\cap T^\s G)\to S\cap TP$ is a 
smooth isomorphism
in every fiber, there exists a unique 
smooth section $\sigma$ of 
$S/(S\cap T^\s G)$ defined on $\s\inv(\bar U)$
such that $\s(\sigma(g))=\bar X(\s(g))$
for all $g\in U$. Choose a representative 
$X\in\Gamma(S)$ for $\sigma$, then we have 
$T_g\s X(g)=\bar X(\s(g))$ for all $g\in U$.
\end{proof}

We say that a vector field $X\in\mx(G)$ is $\tg$- (respectively $\s$-)
descending
if there exists $\bar X\in\mx(P)$ such that
$X\sim_\tg \bar X$ (respectively $X\sim_\s\bar X$),
that is, for all $g\in\dom(X)$, we have 
$T_g\tg X(g)=\bar X(\tg(g))$.
 
\begin{corollary}
Let $G\rr P$ be a Lie groupoid and $S\subseteq TG$ a multiplicative subbundle. Then 
\begin{enumerate}
\item $S$ is spanned by local $\tg$-descending sections and 
\item $S$ is spanned by local $\s$-descending sections.
\end{enumerate}
\end{corollary}
\begin{proof}
Choose $g\in G$ and smooth sections $\bar X_1,\ldots,\bar X_k$
of $S\cap TP$ spanning $S\cap TP$ on a neighborhood
$U_1$ of $\tg(g)$. Choose also $Y_1,\ldots,Y_m$ spanning $(S\cap T^\tg G)\an{P}$
in a neighborhood $U_2$ of $\s(g)$.
The vector fields 
$Y_1^l,\ldots,Y_m^l$ span $S\cap T^\tg G$ on the neighborhood $\s\inv(U_2)$
of $g$ and we find smooth $\tg$-descending sections 
$X_1,\ldots,X_k$ of $S$ such that $X_i\sim_\tg\bar X_i$ on $\tg\inv(U_1)$.
The sections $Y_1^l,\ldots,Y_m^l,X_1,\ldots,X_k$ 
are $\tg$-descending and span $S$ on the neighborhood $U:=\s\inv(U_2)\cap 
\tg\inv(U_1)$ of $g$.
\end{proof}

\medskip

Let $M$ be a smooth manifold and $F\subseteq TM$ a subbundle spanned
by a family $\mathcal F$ of vector fields.
If $F$ is involutive, it is integrable in the sense of Frobenius 
and each of its leaves is an \emph{accessible set}
of $\mathcal F$, i.e.,
the leaf $L_m$ of $F$ through  $m\in M$
is the set 
$$L_m=\left\{\phi^{X_1}_{t_1}\circ\ldots\circ  \phi^{X_k}_{t_k}(m)\left|
\begin{array}{c}
k\in \N, X_1,\ldots, X_k\in\mathcal F, t_1,\ldots,t_k\in \R\\
\text{ and } \phi^{X_i} \text{ is a local flow of } X_i
\end{array}\right.\right\}$$
(see \cite{OrRa04}, \cite{Stefan74a}, \cite{Sussmann73}, \cite{Stefan80}).
If the multiplicative subbundle $S\subseteq TG$ is involutive, we get the following corollary
which will be very useful in the next section.
\begin{corollary}\label{accessible-sets}
Let $G\rr P$ be a Lie groupoid and $S\subseteq TG$ an involutive multiplicative
subbundle.
Then $S$ is completely integrable in the sense of Frobenius and its leaves 
are the accessible sets of each of the two following families of vector fields.
\begin{align}
\mathcal F^\tg_S=&\{X\in\Gamma(S)\mid \exists \bar X\in\mx(P) \text{ such that } 
X\sim_\tg \bar X\}\label{t_accessible}\\
\mathcal F^\s_S=&\{X\in\Gamma(S)\mid \exists \bar X\in\mx(P) \text{ such that } 
X\sim_\s \bar X\}\label{s_accessible}.
\end{align}
\end{corollary}

By Corollary \ref{trelated}, there exists for each section $\bar X$
of $TP\cap S$ defined on $U\subseteq P$ a section 
of $S$ that is defined on $\tg\inv(U)$ and $\tg$-related to $\bar X$. We can find a family of spanning
sections of $S\cap TP$ that are all complete, but the 
corresponding sections of $S$ are then not necessarily complete.
For the proof of our main theorem,
 we will have to assume that $S\cap TP$ is spanned by the following families
of vector fields:
\begin{align}
\bar{\mathcal F}^\tg_S&:=\left\{\bar X\in\Gamma(S\cap TP)
\left|\begin{array}{c}
 \exists X\in\Gamma(S) \text{ such that } 
X\sim_\tg \bar X\\
\text{ and } X, \bar X \text{ are complete},\\
 \dom(X)=\tg\inv(\dom(\bar X))
\end{array}\right.
\right\}\label{complete_STP_tg}\\
\bar{\mathcal F}^\s_S&:=\left\{\bar X\in\Gamma(S\cap TP)
\left|\begin{array}{c}
 \exists X\in\Gamma(S) \text{ such that } 
X\sim_\s \bar X\\
\text{ and } X, \bar X \text{ are complete},\\
 \dom(X)=\s\inv(\dom(\bar X))
\end{array}\right.
\right\}.\label{complete_STP_s}
\end{align}
Note that there exists a complete family $\bar{\mathcal F}^\tg_S$ if and only if there exists 
a complete family $\bar{\mathcal F}^\s_S$, since the vector fields 
in $\mathcal F^\tg_S$ are the inverses of the vector fields in 
$\mathcal F^\s_S$ and vice versa. 
We say that $S$ is \emph{complete} if it has this property.

For instance, if $G\rr P$ is such that the target map $\tg$ (and equivalently the source map $\s$)
 is proper, then it is easy to see that any multiplicative subbundle $S\subseteq TG$
is complete.
Also, note that the multiplicative distributions in Examples \ref{ex-basegp} and \ref{ex-VB}
are complete.

\bigskip

We study now the left invariant vector fields that leave the multiplicative subbundle $S$ invariant.
The following lemma is a result that is shown in \cite{JoOr11}.
\begin{lemma}\label{lem_work_in_progr}
Let $G\rr P$ be a Lie groupoid and $S\subseteq TG$ 
an involutive multiplicative subbundle. Then 
the Bott connection, i.e., the partial $S$-connection on $TG/S$,
induces a  flat partial $S\cap TP$-connection $\nabla$ on $AG/(S\cap AG)\to P$
with the following property. If $a\in  \Gamma(AG)$
is such that its class in $\Gamma\left(AG/(S\cap AG)\right)$
is $\nabla$-parallel, then the left invariant vector field $a^l$ satisfies
\begin{equation}\label{very_useful}
 \left[a^l,\Gamma(S)\right]\subseteq \Gamma(S).
\end{equation}
\end{lemma}
Note that \eqref{very_useful} is automatically satisfied for any 
$a\in\Gamma(AG\cap S)$ since $S$ is involutive and 
$a^l$ is a section of $S\cap T^\tg G$ (see Lemma \ref{constant_rank}).
We can hence say that a section $a\in\Gamma(AG)$ is $\nabla$-parallel
if its class $\bar{a}$ in $\Gamma(AG/(AG\cap S))$
is $\nabla$-parallel.

Note also that \eqref{very_useful} implies that the flow of $a^l$ leaves the subbundle
$S\subseteq TG$ invariant. This fact is well-known, see for instance the appendix of \cite{JoOr11} for a proof of it.

\begin{corollary}\label{cor_of_crucial}
Let $G\rr P$ be a Lie groupoid and $S\subseteq TG$ 
an involutive multiplicative subbundle. Then
$AG$ is spanned by the following family of (local) sections 
\[\mathcal A^S:=\left\{a\in\Gamma(AG)\left| \left[a^l,\Gamma(S)\right]\subseteq \Gamma(S)\right.\right\}.\] 
\end{corollary}
\begin{proof} We show that there 
exists for each point $p\in P$ a local frame 
of $\nabla$-parallel sections for $AG$ defined on a neighborhood of $p$, where 
$\nabla$ is the flat partial $S\cap TP$-connection as in the preceding lemma.
Let $m:=\dim(P)$ and $r:=\operatorname{rank}(AG)$. Choose  a foliated chart domain $U$ centered at $p$ and 
described by
coordinates $(x^1,\dots,x^m)$ such that the first $ k $ among them define the 
local integral submanifold of $TP\cap S$ containing $p$.
Let $\Sigma\subseteq U$ be the slice 
$\phi^{-1}(\{0\}\times \R^{m-k})$, where $\phi:U\to\R^m$ is
the chart adapted to the foliation.

Denote $l:=\operatorname{rank}(AG/(AG\cap S))$.
Choose $ a_1,\ldots, a_l\in\Gamma(AG)$ such that 
 $\bar a_1,\ldots,\bar a_l\in \Gamma(AG/(AG\cap S))$ 
is a  basis frame for $AG/(AG\cap S)$ on $U$. 
We consider this frame at points of $\Sigma\cap U$ and 
construct $\bar \alpha_1,\ldots, \bar \alpha_l\in\Gamma(AG/(AG\cap S))$
as follows. If  $q\in U$, $\phi(q)=(x_1,\ldots,x_m)$, then we find a path
$c:[0,1]\to\phi^{-1}(\R^k\times \{(x_{k+1},\ldots,x_n)\})$
(the leaf of $TP\cap S$ through $q$) 
with $c(1)=q$ and $c(0)=q'\in \Sigma$ 
satisfying $\phi(q')=(0,\ldots,0,x_{k+1},\ldots,x_n)$.
Define
 $\bar{\alpha}_i(q):=P_c^1(\bar a_i(q'))$, where 
$P_c^1(\bar a_i(q'))$ is the parallel translate 
of $\bar a_i(q')$ along $c$ at time $1$ by the $TP\cap S$-partial 
connection $\nabla$.
Since $U$ is simply connected and  $\nabla$ is flat,
parallel translation is independent of the chosen path
(see, for example, \cite{Iliev06}), hence the sections
 $\bar\alpha_i$ are $\nabla$-parallel sections of $AG/(AG\cap S)$
and   form a pointwise basis of $AG/(AG\cap S)$ on $U$.
Choose representatives $\alpha_1,\ldots,\alpha_l\in\Gamma(AG)$
for $\bar \alpha_1,\ldots,\bar \alpha_l$. 
 Take $\alpha_{l+1},\ldots,\alpha_r$ to be a frame of $AG\cap S$ over $U$. Then 
$\alpha_{1},\ldots,\alpha_r$ is a frame of $AG$ over $U$
composed of $\nabla$-parallel sections.
\end{proof}

\subsection{The leaf space of an involutive multiplicative subbundle of $TG$.}
Let $G\rr P$ be a Lie groupoid and $S\subseteq TG$ an involutive 
multiplicative subbundle. Then
$S$ is completely
integrable in the sense of Frobenius.
Let $\pr:G\to G/S$ be the projection to the space of leaves of $S$.

 It is easy to check that
the intersection $S\cap TP$ is an involutive 
subbundle of $TP$ and hence itself
also completely integrable. 
Let $P/S$ be \emph{the space of leaves of $S\cap TP$} in 
$P$ and $\pr_\circ$ the projection $\pr_\circ: P\to  P/S$.
Note that by a theorem in \cite{Jotz11c},
$S$ is involutive if and only if $S\cap TP$ is involutive
and $S\cap AG$ is a subalgebroid of $AG$.
Let $P/S$ be \emph{the space of leaves of $S\cap TP$} in 
$P$ and $\pr_\circ$ the projection $\pr_\circ: P\to  P/S$.

For $g,h\in G$, we will write $g\sim_S h$ 
if $g$ and $h$ lie in the same leaf of $S$
and $[g]:=\{h\in G\mid h\sim_S g\}$ for the leaf of $S$ through $g\in G$. 
By the following proposition,
we can use the same notation for the equivalence
relation defined by the foliation by $S\cap TP$ on $P$.
We will write $[p]_\circ$ for the leaf of $S\cap TP$ through $p\in P$.

Note that $g$ and $h\in G$ lie in the same leaf of $S$ if they can 
be joined by finitely many flow curves of vector fields lying in $\mathcal
F_S^\tg$ or $\mathcal F_S^\s$ (see Corollary \ref{accessible-sets}). 
For simplicity, if $g\sim_S h$, we will often assume 
without loss of generality that $g$ and $h$ can be joined by \emph{one}
such integral curve.

\medskip

We will see that if $S$ is complete with some additional properties,
 then the space $G/S$ inherits a groupoid structure over 
$P/S$. If the foliations are simple, we will
get a \emph{Lie} groupoid 
$G/S\rr P/S$. We start by showing that the structure maps $\epsilon$, $\s$, $\tg$, $\mathsf i$ 
\emph{always} induce maps on $G/S$.

\begin{theorem}\label{epsilon}
Let $G\rr P$ be a Lie groupoid and $S\subseteq TG$ an involutive, multiplicative 
subbundle.
\begin{enumerate} 
\item If $p$ and $q$ lie in the same leaf of $S\cap TP$, then
$p$ and $q$ seen as elements of $G$ lie in the same leaf of $S$.
Hence, the map $[\epsilon]:P/S\to G/S$, 
$\pr\circ\epsilon=[\epsilon]\circ\pr_\circ$, is  well-defined.
\item Choose $g, h\in G$ such that $g\sim_S h$. Then 
$\s(g)\sim_S\s(h)$ and $\tg(g)\sim_S\tg(h)$ and the maps
$[\s], [\tg]:G/S\to P/S$ defined by    
$[\s]\circ\pr=\pr_\circ\circ\,\s$, $[\tg]\circ\pr=\pr_\circ\circ\,\tg$ are well-defined.
\item Choose $g, h\in G$ such that $g\sim_S h$. Then 
$g\inv\sim_S h\inv$. Hence, 
$[\mathsf i]:G/S\to G/S$, $[\mathsf i]\circ \pr=\pr\circ \mathsf i$, is well-defined.
\end{enumerate}
\end{theorem}

\begin{proof}
\begin{enumerate}
\item This is a standard fact about foliations that are compatible with submanifolds.
\item If $g$ and $h\in G$ are in the same leaf of $S$, we find 
by Corollary \ref{accessible-sets} smooth $\s$-descending 
vector fields $X_1,\ldots,X_k\in\Gamma(S)$ and $t_1,\ldots,t_k\in\R$ such that
$h=\phi^k_{t_k}\circ\ldots\circ\phi^1_{t_1}(g)$, where 
$\phi^i$ is the flow of the vector field $X_i$ for each $i=1,\ldots,k$.
There exist then smooth vector fields 
$\bar X_1,\ldots, \bar X_k\in\Gamma(S\cap TP)$ such that 
$X_i\sim_\s\bar X_i$, and hence, 
if $\bar \phi^i$ is the flow of the vector field $\bar X_i$,
$\s\circ\phi^i=\bar\phi^i\circ\s$ for $i=1,\ldots,k$.
We compute then 
\[\s(h)=\s\left(\phi^k_{t_k}\circ\ldots\circ\phi^1_{t_1}(g)\right)
=\bar \phi^k_{t_k}\circ\ldots\circ\bar\phi^1_{t_1}(\s(g)),\]
which shows that $\s(h)$ and $\s(g)$ lie in the same leaf of 
$S\cap TP$. The map $[\s]:G/S\to P/S$, 
$[g ]\mapsto[\s(g)]_\circ$ is consequently
well-defined. We show in the same manner, but using this time
the family \eqref{t_accessible}, that 
 $[\tg]:G/S\to P/S$, 
$[g ]\mapsto[\tg(g)]_\circ$ is well-defined (note that we don't use here
the completeness of $\mathcal F_S^\tg$). 
\item If $g\sim_S h$, then there exists without loss of generality one  
smooth section $X\in\Gamma(S)$ and $\sigma\in\R$ such that
$g=\phi^X_\sigma(h)$. 
Since $X(\phi^X_\tau(h))\in S(\phi^X_\tau(h))$
for all $\tau\in [0,\sigma]$, the curve $c:[0,\sigma]\to G$, $c(\tau)=(\phi^X_\tau(h))\inv$ satisfies
$\dot c(\tau)=T_{\phi^X_\tau(h)}\mathsf i \left(X(\phi^X_\tau(h))\right)\in S(\mathsf i(\phi^X_\tau(h)))$
for all $\tau\in[0,\sigma]$. The image  of $c$ lies hence in the leaf of $S$ through $c(0)=h\inv$.
Since $c(\sigma)=g\inv $, we have shown that $h\inv \sim_S g\inv $.
\end{enumerate}
\end{proof}
Hence, we have shown that the structure maps $\epsilon$, $\s$, $\tg$ and $\mathsf i$ 
project to well-defined maps on $P/S$ and $G/S$
and the following 
diagrams commute.
\begin{align*}
\begin{xy}
\xymatrix{
G\ar[r]^\s\ar[d]_{\pr}&P\ar[d]^{\pr_\circ}\\
G/S\ar[r]_{[\s]}&P/S
}
\end{xy}\qquad \qquad
\begin{xy}
\xymatrix{
G\ar[r]^\tg\ar[d]_{\pr}&P\ar[d]^{\pr_\circ}\\
G/S\ar[r]_{[\tg]}&P/S
}
\end{xy}
\end{align*}
\begin{align*}
\begin{xy}
\hspace*{-0.2cm}\xymatrix{
P\ar[r]^\epsilon\ar[d]_{\pr_\circ}&G\ar[d]^{\pr}\\
P/S\ar[r]_{[\epsilon]}&G/S
}
\end{xy}\qquad\qquad
\begin{xy}
\xymatrix{
G\ar[r]^{\mathsf i}\ar[d]_{\pr}&G\ar[d]^{\pr}\\
G/S\ar[r]_{[\mathsf i]}&G/S
}
\end{xy}
\end{align*}

\subsubsection*{\textbf{On the existence of an induced multiplication}}
%  For the multiplication, 
% we will need the following technical lemmas.
% Note that we have not used the completeness condition until here.

Assume that the leaf space $G/S$ has a groupoid structure over $P/S$ such that
the projection $(\pr,\pr_\circ)$ is a groupoid morphism. 
Then:
\begin{equation*}
\begin{array}{c}
\text{ If } g, h \in G \text{ are such that } \s(g)=\tg(h), \\
\text{ then }  [g]\star[h]=[g\star h].
\end{array}
\end{equation*}
Thus, the multiplication 
on $G/S\,\rr\, P/S$ has to be defined as follows.
\begin{equation}\label{multiplication}
\begin{array}{c}
\text{ If } [g],[h]\in G/S \text{ are such that } [\s]([g])=[\tg]([h]), 
\text{ then the product of } [g] \text{  and } [h]\\
\text{ is given by }  [g]\star [h]=[g'\star h],
\text{ where } g'\in G \text{ is such that  } \\
g\sim_Sg' \text{ and } \s(g')=\tg(h). \\
\end{array}
\end{equation}

For this to be well-defined, we will have to assume that the following conditions
are satisfied. 
\begin{align}
\text{For all } g\in G,\, p\in P \text{ such that  } p\sim_S\s(g), \text{ there exists }
h\in G\nonumber\\ 
\text{ such that } h\sim_S g \text{ and } \s(h)=p.\label{condition1}\\
\text{For all } g\in G \text{ and  } p:=\tg(g):
\left([p]\cap\s\inv(p)\right) \star g
= [g]\cap \s\inv(\s(g)).\label{condition}
\end{align}

In Example \ref{example_of_marco} (taken from  \cite[Example 6]{Zambon10}),
 we have a Lie groupoid with a multiplicative foliation such that 
\eqref{condition} is \emph{not} satisfied, and the leaf space does not inherit a multiplication.

\medskip

Condition \eqref{condition} is satisfied for instance if all the intersections
$[g]\cap\tg\inv(\tg(g))$ have one connected component 
and are hence the leaves  of the involutive vector bundle
$T^\tg G\cap S$. This is the case if the Lie groupoid is a Lie group, and also
in Examples \ref{ex-basegp} and \ref{ex-VB}.

We will show later that, under some
(rather strong) regularity conditions on the family 
$\mathcal A^S$, \eqref{condition1} and \eqref{condition} hold
(see Proposition \ref{annoying}).

In the next lemma, we prove that \eqref{condition1} is always satisfied if $S$ is complete.

\begin{lemma}\label{existence_of_gH}
Let $G\rr P$ be a Lie groupoid and $S\subseteq TG$ be a \emph{complete} multiplicative 
involutive subbundle of $TG$. If $g\in G$ and $\tg(g)\sim_S p\in P$,
then there exists $h\in G$ such that $g\sim_S h$ and $\tg(h)=p$.
In the same manner, if $\s(g)\sim_S p\in P$,
then there exists $h\in G$ such that $g\sim_S h$ and $\s(h)=p$.
That is, condition \eqref{condition1} is satisfied.
\end{lemma}
\begin{proof}
Choose a vector field $\bar X\in\bar{\mathcal F}^\tg_S$ and $\sigma\in\R$ such that
$p=\phi_\sigma^{\bar X}(\tg(g))$. We find then a $\tg$-descending vector field 
$X\sim_\tg \bar X$ defined at $g$.
Since $\bar X$ and $X$ can be taken complete by hypothesis,
the integral curve  of $X$ starting at $g$ is defined at $\sigma$ and
 we have  $\tg(\phi^X_\tau(g))=\phi_\tau^{\bar X}(\tg(g))$ for all 
$\tau\in[0,\sigma]$. Set $h:=\phi^X_\sigma(g)$, then 
$h\sim_S g$ and $\tg(h)=\phi_\sigma^{\bar X}(\tg(g))=p$.
\end{proof}

We can now formulate our main theorem and complete its proof. 
\begin{theorem}\label{groupoid_leaf_space}
Let $G\rr P$ be a  Lie groupoid
and $S$ a complete, multiplicative, involutive subbundle 
of $TG$. 
Then there is an induced groupoid structure
on the leaf space $G/S\rr P/S$, such that 
$(\pr,\pr_\circ)$ is a groupoid morphism, if and only if
\eqref{condition} is satisfied.
\end{theorem}
\begin{remark}
 Note that in general, the induced groupoid $G/S\rr P/S$ is not a Lie groupoid.
As in the easier Lie group case, we get a Lie groupoid if the 
involutive subbundles $S\subseteq TG$ and $S\cap TP\subseteq TP$ are simple.

Note that if a \emph{Lie group} $G$ with Lie algebra $\lie g$ is simply connected, then 
any ideal $\lie s\subseteq \lie g$  integrates automatically to a closed, normal subgroup (\cite{HiNe91}).
Thus, the space of leaves of a multiplicative subbundle $S:=\lie s^L=\lie s^R\subseteq TG$
is always a \emph{Lie} group. In the general case of a $\tg$-simply connected Lie groupoid, 
it is easy to see that
  the leaf  space of a multiplicative foliation is not necessarily a manifold. Take for instance
any manifold as a groupoid over itself  (and hence $\tg$-simply connected) and any non-simple foliation
 (that is automatically multiplicative) on the manifold.
\end{remark}

\begin{proof}[Proof of Theorem \ref{groupoid_leaf_space}]
Assume that
 $G/S\,\rr\, P/S$ is a groupoid such that $(\pr,\pr_\circ)$ is a groupoid morphism and choose $g\in G$.
If $h\in [g]\cap\s\inv(\s(g))$, then 
we have $[h\star g\inv]=[h]\star[g\inv]=[g]\star[g\inv]
=[\tg(g)]$
and hence 
\[h=(h\star g\inv)\star g\in ( [\tg(g)]\cap \s\inv(\tg(g)))\star g.\]
The converse inclusion in \eqref{condition} can be checked in the same manner.

\medskip

Conversely, assume 
that  \eqref{condition} holds.
We show that the multiplication \eqref{multiplication} is well-defined, i.e., that is doesn't depend 
on the choice of the representatives $g,h$.
Using Theorem \ref{epsilon}, the groupoid axioms are then easily verified to hold
and the projection $(\pr,\pr_\circ)$ is a groupoid morphism by definition of the structure maps 
of $G/S\,\rr\, P/S$.

Choose $[g]$ and $[h]\in G/S$ such that $[\s]([g])=[\tg]([h])$. Then 
we have $[\s(g)]=[\tg(h)]$ by definition of $[\s]$ and $[\tg]$ and,  by
Lemma \ref{existence_of_gH}, there 
exists $g'\in [g]$ such that $\s(g')=\tg(h)$. For simplicity, assume that 
the chosen representative $g$ already satisfies this condition. 

Assume that $g,g'$ are two such choices, i.e., $g'\in[g]$ and $\s(g')=\tg(h)=\s(g)$.
By \eqref{condition}, we have
\[[g\star h]\cap \s\inv(\s(h))= \left([\tg(g)]\cap\s\inv(\tg(g))\right) \star (g\star h)
= \left([g]\cap \s\inv(\s(g))\right)\star h,
\]
and hence
$g\star h \sim_S g'\star h$. That is, 
\eqref{multiplication} doesn't depend on the choice of the representative $g$
of $[g]$. 
By symmetry, \eqref{multiplication} doesn't depend on the choice of $h$.
\end{proof}

\begin{example}\label{example_of_henrique}
Consider a $\tg$-connected Lie groupoid $G\rr P$
and let a connected Lie group $H$ act freely and properly on $G\rr P$
  by Lie groupoid 
morphisms. Let $\Phi:H\times G\to G$ be the action. 
That is, for all $h\in H$, 
the map $\Phi_h:G\to G$ is a groupoid morphism
over the map $\phi_h:=\Phi_h\an{P}:P\to P$.
Let $\V\subseteq TG$ be the vertical space 
of the action, i.e., 
$\V(g)=\{\xi_G(g)\mid \xi\in\lie h\}$
for all $g\in G$, where $\lie h$ is the Lie algebra of $H$.

We check that $\V\subseteq TG$ is multiplicative.
Choose $\xi_G(g)\in \V(g)$. Then we have 
\begin{align*}
T_g\tg\left(\xi_G(g)\right)&=\left.\frac{d}{dt}\right\an{t=0}\tg(\Phi_{\exp(t\xi)}(g))
=\left.\frac{d}{dt}\right\an{t=0}\phi_{\exp(t\xi)}(\tg(g))\\
&=\xi_P(\tg(g))\quad \in\quad\V(\tg(g))\cap T_{\tg(g)}P
\end{align*}
and in the same manner
$$T_g\s\left(\xi_G(g)\right)\in\V(\s(g))\cap T_{\s(g)}P.$$
This shows  that $\V\cap T^\s G=\V\cap T^\tg G=\{0\}$
in this example.

If $\xi_G(g)\in\V(g)$ and $\eta_G(g')\in\V(g')$
are such that 
$$T_g\s\left(\xi_G(g)\right)=T_g\tg\left(\eta_G(g')\right),$$
then we have $\s(g)=\tg(g')=:p$ and 
$$\xi_P(p)=\eta_P(p),$$
which implies $\xi=\eta$ since the action is free. We get then
\begin{align*}\xi_G(g)\star \eta_G(g')&=\xi_G(g)\star \xi_G(g')
=\left.\frac{d}{dt}\right\an{t=0}\Phi_{\exp(t\xi)}(g)\star \Phi_{\exp(t\xi)}(g')\\
&=\left.\frac{d}{dt}\right\an{t=0}\Phi_{\exp(t\xi)}(g\star g')
=\xi_G(g\star g')\in\V(g\star g').
\end{align*}
The inverse of $\xi_G(g)$ is then $\xi_G(g\inv)$ for all $\xi\in\lie h$ and
$g\in G$. 

Choose $g\in G$. Then we have 
\[[g]\cap\s\inv(\s(g))=\{g'\in G\mid \s(g)=\s(g'),\, \exists\, x\in H: x\cdot g=g'\}=\{g\},
\]
since $x\cdot g=g'$ implies $x\cdot\s(g)=\s(g')=\s(g)$ and hence 
$x=e$ because the action of $H$ on $P$ is assumed to be free. As a consequence, Condition \eqref{condition}
\[[g]\cap\s\inv(\s(g))=\left([\tg(g)]\cap\s\inv(\tg(g))\right)\star g
\]
is satisfied in a trivial manner.

It follows also from the considerations above that the vector fields $\xi_G$ are $\tg$-
and $\s$-descending 
to $\xi_P$. Hence, $\V$ is spanned by complete vector fields, that are $\tg$-related
to complete vector fields.

Hence, we recover from Theorem \ref{groupoid_leaf_space}
the fact that
the quotient $G/H\rr P/H$ has the structure of a Lie groupoid
such that the projections 
$\pr:G\to G/H$, $\pr_\circ:P\to P/H$
form a Lie groupoid morphism. 
\end{example}

\begin{remark}\label{geom_mech_remark}
In discrete mechanics, the velocity space $TQ$ of a configuration space 
$Q$ is replaced by $Q\times Q$.
The lift of a Lie group action by a Lie group $H$ on the configuration space is then naturally 
replaced by the diagonal action of $H$ on $Q\times Q$.
 It is observed by 
\cite{Weinstein96} that the Lagrangian formalism
for discrete mechanics can be generalized using groupoids.
This idea has been developped by \cite{MaMaMa06} and \cite{IgMaMaMa08}. In this setting, 
Lie group actions preserving the geometric structure should be 
Lie group actions by Lie groupoid morphisms as in the previous example.
This suggests that Theorem \ref{groupoid_leaf_space} could be  very useful
for reduction by distributions in the context of discrete mechanics and discrete mechanics 
on Lie groupoids.
\end{remark}

\bigskip
We show now that under
 additional conditions on the family $\mathcal A^S$, 
\eqref{condition1} and \eqref{condition} are satisfied if the 
 Lie groupoid is $\tg$-connected.
We will assume that:
\begin{itemize}
\item The elements of $\mathcal A^S$ are globally defined. 
\item The images of the elements of $\mathcal A^S$ under the anchor are complete.
This implies that the corresponding left invariant vector fields 
on $G$ are also complete (see \cite[Theorem 3.6.4]{Mackenzie05}). 
\end{itemize}
\begin{proposition}\label{annoying}
 Let $G\rr P$ be a $\tg$-connected Lie groupoid and $S\subseteq TG$ an  involutive, multiplicative 
subbundle. Assume that $\mathcal A^S$ is a family of \emph{global} sections of $AG$, with \emph{complete} 
images under the anchor. 
Then \eqref{condition1} and  \eqref{condition} are satisfied and there is hence an induced groupoid structure
on the leaf space $G/S\rr P/S$, such that 
$(\pr,\pr_\circ)$ is a groupoid morphism.
\end{proposition}

\begin{proof}
Since $AG$ is spanned by $\mathcal A^S$ and 
 $G\rr P$ is $\tg$-connected, 
each element $h$ of $G$ can  be written 
\[ h=R_{\Exp(t_1a_1)}\circ \ldots\circ R_{\Exp(t_na_n)}(\tg(h))\]
with $n\in\N$, $a_1,\ldots,a_n\in\mathcal A^S$ and $t_1,\ldots,t_n\in \R$.
Without loss of generality, we can assume in the following that $n=1$ and we set $a_1=:a$, $t_1=:t$.
That is, $h=\Exp(ta)(\tg(h))$. By hypothesis, for all $t\in \R$, the bisection $\Exp(ta)$ is globally
defined on $P$.

Equation \eqref{very_useful} implies that the flow of $a^l$, which is 
equal to $R_{\Exp(\cdot a)}$, leaves $S$ invariant. As a consequence,
if $g'\sim_S g$, then 
$ R_{\Exp(\tau a)}(g')\sim_S  R_{\Exp(\tau a)}(g)$ for all $\tau$.

In particular,  if $q\in P$ is  $S\cap TP$-related to $\tg(h)$, then 
$h=R_{\Exp(t a)}(\tg(h))\sim_S  R_{\Exp(t a)}(q)=\Exp(ta)(q)$
and $\Exp(ta)(q)$ is such that $\tg\left(\Exp(ta)(q)\right)=q$. This implies \eqref{condition1}, by using the inversion.

Furthermore, if $\s(g')=\s(g)=\tg(h)$, we get 
$g'\star h=R_h(g')= R_{\Exp(t a)}(g')\sim_S R_{\Exp(t a)}(g)=R_h(g)=g\star h$.
\end{proof}

We end this paragraph with an example of a multiplicative foliation which leaf space is not a groupoid.

\begin{example}[\cite{Zambon10}]\label{example_of_marco}
Let $P$ be the trivial circle bundle over the open $2$-disk $D$, but with one
point removed in the fiber over $0 \in D$. We write $P= \hat\sfe^1\to D$ , where
$\hat\sfe^1_p$ denotes the circle for all non-zero $p \in D$, while 
$\hat\sfe^1_0$ is the circle with a point removed.
 Note that $\pi_1(P) = \mathbb Z$, generated by any of the circle fibers. It is easy to
see that the universal cover of $P$ is $\tilde P = (D \times \R)\setminus(\{0\} \times \mathbb Z)$. 
We write $\tilde P = \hat \R\to D$ as a bundle over $D$, where $\hat\R_p =\R$ for non-zero $p \in D$
and $\hat \R_0 = \R\setminus\mathbb Z$.

Consider the Lie groupoid $G:=\hat\R \times_{\mathbb Z} \hat \R\,\rr\, P$, where 
the action of $\mathbb Z=\pi_1(P)$ is the diagonal action by deck transformations.
Let $\pi$ be the projection $\hat\R \times\hat \R \to \hat\R \times_{\mathbb Z} \hat \R$, 
$(p_1,x_1,p_2,x_2)$ coordinates on $\hat\R \times \hat\R$ and $(p,x)$ coordinates on $P$. The subbundle 
of $T(\hat\R \times \hat\R)$ spanned by $\partial_{x_1}$ and $\partial_{x_2}$
is easily seen to project under $\pi$ to a multiplicative subbundle $S$  of $TG$.
The intersection $S_P:=S\cap TP$ 
is the span of $\partial_x$ and we find $P/S_P\simeq D$.

The leaves of $S$ almost coincide with the
fibers of the natural projection onto $D \times D$: for $(p_1,p_2)\in D\times D$, the leaf $L_{[p_1,x_1,p_2,x_2]}$
of $S$ through $g=[p_1,x_1,p_2,x_2]\in G$ is:
\begin{equation*}
\left\{\begin{array}{l}
 L_{g}=\{[(p_1 ,t_1 , p_2 ,t_2 )] \mid t_1, t_2 \in \R\} 
\quad \text{  if }p_1\neq 0 \text{ and } p_2\neq 0,\\
 L_{g}=\{[(p_1 ,t_1 , p_2 ,t_2 )] \mid t_1\in\R\setminus\mathbb Z, t_2 \in \R\} 
\quad\text{  if }p_1=0 \text{ and } p_2\neq 0,\\
L_{g}=\{[(p_1 ,t_1 , p_2 ,t_2 )] \mid t_1\in \R, t_2 \in \R\setminus\mathbb Z\} 
\quad\text{  if }p_1\neq 0 \text{ and } p_2=0,\\
L_{g}=\{[(p_1 ,t_1 , p_2 ,t_2 )] \mid t_1\in ([x_1],[x_1]+1), t_2 \in ([x_2],[x_2]+1)\} 
\quad\text{  if }p_1=p_2=0.
\end{array}\right.
\end{equation*}
Thereby, for $x\in\R$, $[x]$ is the biggest integer smaller than $x$.
In other words, the leaf through $[p_1,x_1,p_2,x_2]$ is a cylinder
if $p_1\neq 0$  and  $p_2\neq 0$, a rectangle 
if either $p_1$ or $p_2$ vanishes, and, 
 over $(0, 0)\in D \times D$, we
have the quotient of $(\R \setminus\mathbb Z) \times (\R \setminus \mathbb Z)$
 by the diagonal $\mathbb Z$-action, which consists of
countably many leaves. Hence the leaf space is $\widehat{D\times D}$,
where the latter denotes the non-Hausdorff manifold obtained from $D \times D$ replacing
$(0, 0)$ with a copy of $\mathbb Z$.

Is it easy to see that, as shown in Theorem \ref{epsilon}, there are
 well-defined ``source'' and ``target''
 maps $\widehat{D\times D}\to D$ such that 
\begin{align*}
\begin{xy}
\xymatrix{
G=\hat\R \times_{\mathbb Z} \hat\R\ar[r]^{\pr}
\ar@<.6ex>[d]^\s\ar@<-.6ex>[d]_\tg& \widehat{D\times D}\ar@<.6ex>[d]^{[\s]}\ar@<-.6ex>[d]_{[\tg]}\\
P=\hat\sfe\ar[r]_{\pr_\circ}&D }
\end{xy}
\end{align*}
commutes, but the multiplication
does not descend to the quotient.
 Indeed, consider $(0, p) \in \widehat{D\times D}$ where $p$ is
non-zero. A preimage under $\pr$ is $[(0, \mu_1 ), (p, x )]$ where  $\mu_1 \in\R\setminus\mathbb Z$ and $x\in \R$
are arbitrary. Similarly, we consider $(p, 0)\in\widehat{D\times D}$ and as a preimage we pick
$[(p, x ), (0, \mu_2)]$ where again  $\mu_2 \in\R\setminus\mathbb Z$ is arbitrary. Now multiplying these two
elements of $G$ we obtain $[(0, \mu_1 ), (0, \mu_2)]$. The value of its 
projection under $\pr$ depends on the concrete choice of $\mu_1$ and $\mu_2$. This shows that
$\widehat{D\times D}$ does not have an induced groupoid structure. In the same manner, one can see that \eqref{condition}
is not satisfied.

\medskip

The Lie algebroid of $G\rr P$ is the projection under $T\pi$ of the restriction to $\Delta_{\hat\R}$ of the subbundle 
$\{0\}\times T\hat\R$ of $T(\hat\R\times\hat\R)=T\hat\R\times T\hat\R$.
The vector fields on the second copy of $D$ and $\partial_{x_2}$ 
are $\pi$-related in this manner to $\nabla$-parallel sections of $AG$.
It is easy to see using points as above that the conditions for Proposition
\ref{annoying} are not satisfied here.
\end{example}

\bigskip
\subsubsection*{\textbf{Tangent and cotangent spaces of the quotient}}
We end this section with the  following lemma, which will be useful for the proof
of Theorem \ref{main}. The proof is a straightforward computation that is left to the reader.
\begin{lemma}\label{tangent_cotangent_structures}
In the setting of Theorem \ref{groupoid_leaf_space}, if 
$G/S\rr P/S$ is a Lie groupoid and 
$(\pr,\pr_\circ)$ a pair of smooth surjective submersions, choose
$v_{[g]}\in T_{[g]}(G/S)$
and $v_{[h]}\in T_{[h]}(G/S)$  such that 
$T_{[g]}[\s] v_{[g]}=T_{[h]}[\tg] v_{[h]}$.
Then we can assume without loss of generality that $\s(g)=\tg(h)$.
If
$v_g\in T_gG$ and $v_h\in T_hG$ are such that 
$T_g\pr v_g=v_{[g]}$ and $T_h\pr v_h=v_{[h]}$, then there exists
$w_g\in S(g)$ such that $T_g\s(v_g-w_g)=T_h\tg v_h$.
We have then $T_{g\star h}\pr((v_g-w_g)\star v_h)=v_{[g]}\star v_{[h]}$.

If $\alpha_{[g]}\in T_{[g]}^*(G/S)$
and $\alpha_{[h]}\in T_{[h]}^*(G/S)$  are  such that 
$\hat{[\s]}(\alpha_{[g]})=\hat{[\tg]}(\alpha_{[h]})$, then 
$$\hat\s((T_g\pr)^*\alpha_{[g]})=(T_{\s(g)}\pr)^*\hat\s(\alpha_{[g]}), \qquad 
\hat\tg((T_h\pr)^*\alpha_{[h]})=(T_{\tg(h)}\pr)^*\hat\tg(\alpha_{[h]}),$$
hence 
$\hat\s\left((T_g\pr)^*\alpha_{[g]}\right)
=\hat\tg\left((T_h\pr)^*\alpha_{[h]}\right)$
and we have 
$$((T_g\pr)^*\alpha_{[g]})\star( (T_h\pr)^*\alpha_{[h]})
=(T_{g\star h}\pr)^*(\alpha_{[g]}\star \alpha_{[h]}).$$
\end{lemma}

\section{An application} 
We use Theorem \ref{groupoid_leaf_space}
to study how a theorem in \cite{Ortiz08} (see also \cite{Jotz11a})
about the shape of \emph{Dirac Lie groups}
generalizes to Dirac groupoids.

\bigskip

The Pontryagin bundle $\mathsf{P}_M=TM
\times_M T^* M$ of a smooth manifold $M$ is
endowed with the  non-degenerate symmetric fiberwise 
bilinear form of signature\linebreak $(\dim M, \dim M)$ given by
\begin{equation}
\langle(v_m,\alpha_m), (w_m,\beta_m)\rangle=\alpha_m(w_m)+\beta_m(v_m)
\label{sym_bracket}
\end{equation}
for all $m\in M$, $v_m,w_m\in T_mM$ and $\alpha_m,\beta_m\in T_m^*M$. An
\emph{almost Dirac structure} (see \cite{Courant90a}) on $M $ is a Lagrangian vector 
subbundle $\mathsf{D} \subset \mathsf{P}_M $. That is, $ \mathsf{D}$ coincides with its
orthogonal relative to \eqref{sym_bracket} and so its fibers are necessarily $\dim M $-dimensional.

\medskip

Let $(M,\mathsf D)$ be an almost Dirac manifold. 
For each $m\in M$, the almost Dirac structure $\mathsf{D}$ 
defines a subspace
$\mathsf{G_0}(m)\subset T_mM $ by 
\[
\mathsf{G_0}(m):= \{v_m \in T_mM \mid (v_m, 0)\in\mathsf D(m) \}.\]
The distribution $\mathsf{G_0}=\cup_{m\in M}\mathsf{G_0}(m)$
is not necessarily smooth.

The almost Dirac structure $\mathsf D$ is 
\emph{a Dirac structure} if its set of sections is closed under the
 \emph{Courant-Dorfman bracket}, i.e.,
\begin{equation}\label{Courant_bracket}
[(X, \alpha), (Y, \beta) ]  
= \left( [X, Y],  \boldsymbol{\pounds}_{X} \beta - \ip{Y} \dr\alpha \right)
\in\Gamma(\mathsf D)
\end{equation} 
for all $(X,\alpha), (Y,\beta)\in\Gamma(\mathsf D)$.
The class of
Dirac structures generalizes the one of Poisson manifolds in the sense
of the following example.
\begin{example}\label{exPoisson}
Let $M$ be a smooth manifold endowed with a globally defined bivector field
$\pi\in\Gamma\left(\bigwedge^2 TM\right)$.  Then  $\mathsf
D_\pi\subseteq \mathsf P_M$ defined by
\[\mathsf{D}_\pi(m)=\left\{(\pi^\sharp(\alpha_m),\alpha_m)\mid
  \alpha_m\in T_m^*M \right\}\quad \text{ for all }\, m\in M,
\]
where $\pi^\sharp:T^*M\to TM$ is defined by 
$\pi^\sharp(\alpha)=\pi(\alpha,\cdot)\in\mx(M)$ for all
$\alpha\in\Omega^1(M)$,
is an almost Dirac structure on $M$. It is a Dirac structure if and only if the Schouten bracket
of the bivector field with itself 
vanishes,
$[\pi,\pi]=0$, that is, if and only if $(M,\pi)$ is a Poisson manifold.
\end{example}

Let $(M, \mathsf D_M)$ and $(N, \mathsf D_N)$ be
two (almost) Dirac manifolds and $\Phi:M\to N$ a smooth map.
Then $\Phi$ is a \emph{forward Dirac map}
if for all  $n\in N$, $m\in \Phi\inv(n)$ and 
$(v_n, \alpha_n)\in\mathsf D_N(n)$ 
there exists $(v_m, \alpha_m)\in\mathsf D_M(m)$
such that $T_m\Phi v_m=v_n$ and $\alpha_m=(T_m\Phi)^*\alpha_n$.

\medskip

Let $G\rr P$ be a Lie groupoid. Recall that the Pontryagin  bundle of $G$ has then a Lie groupoid
structure over $TP\times_PA^*G$.
\begin{definition}[\cite{Ortiz08t}]
An (almost) Dirac  groupoid is a Lie groupoid $G\rightrightarrows P$
endowed with an (almost) Dirac structure $\mathsf D_G$ such that
$\mathsf D_G\subseteq TG\times_G T^*G$ is a Lie subgroupoid.
The (almost) Dirac structure is then said to be \emph{multiplicative}.
\end{definition}

If $G$ is a Lie group, i.e., with $P=\{e\}$, an (almost) Dirac structure $\mathsf D_G$
on $G$ is multiplicative if the multiplication 
$\mathsf m: (G\times G, \mathsf D_G\oplus\mathsf D_G)\to (G,\mathsf D_G)$
is a forward Dirac map. It is shown in \cite{Ortiz08}, \cite{Jotz11a}, 
that the induced distribution $\mathsf{G_0}$
on $G$ is multiplicative, hence left and right invariant and 
equal to $\mathsf{G_0}=\lie g_0^L=\lie g_0^R$ for some 
ideal $\lie g_0\subseteq \lie g$ in the Lie algebra $\lie g$
of $G$. Thus, $\mathsf{G_0}$ is  automatically an involutive subbundle of $TG$ for any
 multiplicative almost Dirac structure on the Lie group $G$. The leaf $N$ of $\mathsf{G_0}$
through $e$ is then a normal subgroup of $G$ and if it is closed, 
the quotient $G/\mathsf{G_0}=G/N$ is a Lie group
such that the projection 
$q:G\to G/N$ is a smooth surjective submersion. If 
$(G,\mathsf D_G)$ is a Dirac Lie group and $N\subseteq G$
is a closed subgroup, 
the quotient $G/N$ inherits a Poisson structure 
$\pi$ such that 
$q:(G,\mathsf D_G)\to (G/N,\pi)$ is a forward Dirac map and $(G/N,\pi)$ is a \emph{Poisson Lie group}.
Furthermore, $  (G/N,\pi)$ is a Poisson homogeneous space of the 
Dirac Lie group $(G,\mathsf D_G)$ (\cite{Jotz11a}).

\medskip

The fact that the characteristic distribution of an almost Dirac Lie group
is always multiplicative is a special case of the fact that the 
 characteristic distribution of an almost Dirac groupoid is always multiplicative, as shows the next proposition.
\begin{proposition}\label{G_0_multiplicative}
Let $(G\rr P,\mathsf D_G)$ be an (almost) Dirac groupoid. Then the
subbundle $\mathsf{G_0}\subseteq TG$ is a (set) subgroupoid over $TP\cap\mathsf{G_0}$. 
\end{proposition}
\begin{proof}
Choose a composable pair $(g,h)\in G\times_PG$ and $v_g\in\mathsf{G_0}(g)$, 
$v_h\in\mathsf{G_0}(h)$ such that 
$T_g\s(v_g)=T_h\tg(v_h)$. Then we have $(v_g, 0_g)\in\mathsf D_G(g)$, 
$(v_h, 0_h)\in\mathsf D_G(h)$
such that 
$\TT\s(v_g,0_g)=(T_g\s(v_g),0_{\s(g)})=(T_g\tg(v_h),0_{\tg(h)})=\TT\tg(v_h,0_h)$
 and hence 
$\TT\tg(v_g,0_g)\in\mathsf D_G(\tg(g))$, 
$\TT\s(v_g, 0_g)\in\mathsf D_G(\s(g))$,
$(v_g,0_g)\inv\in\mathsf D_G(g\inv)$
and $(v_g,0_g)\star(v_h,0_h)\in\mathsf D_G(g\star h)$.
Since $(v_g,0_g)\inv=(v_g\inv,0_{g\inv})$
and $(v_g,0_g)\star(v_h,0_h)=(v_g\star v_h, 0_{g\star h})$, this shows that
$T_g\s(v_g)\in\mathsf{G_0}(\s(g))$, 
$T_g\tg(v_g)\in\mathsf{G_0}(\tg(g))$,
$v_g\inv\in\mathsf{G_0}(g\inv)$ and 
$v_g\star v_h\in\mathsf{G_0}(g\star h)$.
\end{proof}

Yet, the situation in the general case of an (almost) Dirac groupoid
is more involved than in the Lie groups case. First of all, the induced distribution $\mathsf{G_0}$ has not automatically 
constant rank on $G$ anymore. If it is smooth, it is in general not completely integrable in the sense
of Stefan and Sussmann, unless  $\mathsf D_G$ is a Dirac structure.
For instance, each manifold 
can be seen as a (trivial) groupoid over itself 
(i.e., with $\tg=\s=\Id_M$) and any (almost) Dirac manifold 
can thus be seen as an (almost) Dirac groupoid, which will, 
in general not satisfy these conditions. 
Thus,  trivial Dirac groupoids
yield already many examples of Dirac groupoids for which $\mathsf{G_0}$
is not an involutive subbundle of $TG$.
(The class of Dirac groupoids described in the next theorem
seems hence to be a small class of examples.)

If $\mathsf{G_0}$ associated to a
\emph{Dirac groupoid} $(G\rr P, \mathsf D_G)$ is \emph{assumed} to be a vector bundle on $G$, 
then we are in the same situation as in the group case.

We know by Theorem \ref{groupoid_leaf_space} that 
we have to assume that $\mathsf{G_0}$ is \emph{complete} and satisfies
 \eqref{condition} to ensure that the leaf space $G/\mathsf{G_0}$ 
inherits a multiplication map and hence to show that there  exists a Poisson \emph{groupoid} structure 
on the quotient $G/\mathsf{G_0}$.

The next theorem shows that, under these
conditions, we obtain the same result as the one that is 
already known in the group case.

\begin{theorem}\label{main}
Let $(G\rr P, \mathsf D_G)$ be a Dirac groupoid. 
Assume 
that $\mathsf{G_0}$ is a  subbundle of $TG$, that it is complete and that \eqref{condition} 
holds.
If the leaf spaces $G/\mathsf{G_0}$ and 
$P/\mathsf{G_0}$
have smooth manifold structures 
such that the projections are submersions,
then there is an induced multiplicative   Poisson  structure
on the Lie groupoid $G/\mathsf{G_0}\rightrightarrows P/\mathsf{G_0}$, such that
the projection $\pr:G\to G/\mathsf{G_0}$ is a forward Dirac map.
\end{theorem}
\begin{proof}
Recall that  since  $\mathsf D_G$ is assumed to be a Dirac structure, the subbundle
 $\mathsf{G_0}\subseteq TG$ is involutive.
Since $\mathsf{G_0}$ is multiplicative 
by Proposition \ref{G_0_multiplicative}
and all the hypotheses for Theorem \ref{groupoid_leaf_space}
are hence satisfied, we get that $G/\mathsf{G_0}\rr P/\mathsf{G_0}$
has the structure of a Lie groupoid such that if 
$\pr:G\to G/\mathsf{G_0}$ and $\pr_\circ:P\to P/\mathsf{G_0}$
are the projections, then
$(\pr,\pr_\circ)$ is a Lie groupoid morphism.

We have $\mathsf D_G\cap(TG\times_G\mathsf{G_0}^\circ)
=\mathsf D_G\cap(TG\times_G\mathsf{P_1})=\mathsf D_G$ and 
$[\Gamma(\mathsf{G_0}\times_G\{0\}), \Gamma(\mathsf D_G)]\subseteq \Gamma(\mathsf
D_G)$ because $\mathsf D_G$ is Dirac.
Hence, by a result in \cite{Zambon08} (see also \cite{JoRaZa11}), we find that 
the Dirac structure pushes-forward to the quotient $G/\mathsf{G_0}$.
The (almost) Dirac structure $\pr(\mathsf D_G)$  is given by
\[\pr(\mathsf D_G)([g])=
\left\{(v_{[g]},\alpha_{[g]})\in\mathsf P_{G/\mathsf{G_0}}([g])\left|
\begin{array}{c}
\exists v_g\in T_gG \text{ such that }\\
(v_g, (T_g\pr)^*\alpha_{[g]})\in\mathsf D_G(g)\\
\text{ and } T_g\pr v_g=v_{[g]}
\end{array}\right.\right\}\]
for all $g\in G$.
The fact that $\pr(\mathsf D_G)$ is a Dirac structure follows $\mathsf D_G$ being a Dirac structure.
If $(v_{[g]},0)\in \pr(\mathsf D_G)([g])$, there exists 
$v_g\in T_gG$ such that $T_g\pr v_g=v_{[g]}$
and $(v_g,0)\in\mathsf D_G(g)$. But then we get $v_g\in\mathsf{G_0}(g)$
and hence $v_{[g]}=T_g\pr v_g=0_{[g]}$. This shows that
the characteristic distribution $\mathsf{G_0}$
associated to the Dirac structure $\pr(\mathsf D_G)$ is trivial, and 
since it is integrable, $\pr(\mathsf D_G)$ is the graph
of the vector bundle homomorphism $T^*(G/\mathsf{G_0})\to T(G/\mathsf{G_0})$
associated to a Poisson bivector on $G/\mathsf{G_0}$.

We have then to show that the Dirac structure $\pr(\mathsf D_G)$
on $G/\mathsf{G_0}$ is multiplicative.  
Choose  $(v_{[g]}, \alpha_{[g]})
\in \pr(\mathsf D_G)([g])$
and $(v_{[h]}, \alpha_{[h]})\in \pr(\mathsf D_G)([h])$
such that $\TT\s(v_{[g]}, \alpha_{[g]})=\TT\tg(v_{[h]}, \alpha_{[h]})$. 
We can then assume without loss of generality that $\s(g)=\tg(h)$
(see Lemma \ref{tangent_cotangent_structures}).
By the definition of $\pr(\mathsf D_G)$, we find 
then $v_g\in T_gG$ and $v_h\in T_hG$
such that $T_g\pr v_g= v_{[g]}$, 
$T_h\pr v_h=v_{[h]}$ and 
$(v_g, (T_g\pr)^*\alpha_{[g]})\in\mathsf D_G(g)$, 
$(v_h, (T_h\pr)^*\alpha_{[h]})\in\mathsf D_G(h)$.
Since $$\TT\s(v_g,(T_g\pr)^*\alpha_{[g]})\in\mathsf D_G(\s(g))$$
and $$T_{\s(g)}\pr(T_g\s v_g)=T_{[g]}\s v_{[g]},$$
$$(T_{\s(g)}\pr)^*(\s(\alpha_{[g]}))=\hat\s((T_g\pr)^*\alpha_{[g]}),$$ we find that 
$$\TT\s(v_{[g]}, \alpha_{[g]})\in \pr(\mathsf D_G)(\s[g]).$$ In the same manner, 
we get that 
$\TT\tg(v_{[g]}, \alpha_{[g]})\in \pr(\mathsf D_G)(\tg[g])$ and the Dirac structure 
is closed under the source and target maps on $\mathsf P_{G/\mathsf{G_0}}$.

By Lemma \ref{tangent_cotangent_structures}, we find 
$w_g\in \mathsf{G_0}(g)$ such that $T_g\s(v_g-w_g)=T_h\tg v_h$
and $T_{g\star h}\pr((v_g-w_g) \star v_h)=v_{[g]}\star v_{[h]}$. 
By the same Lemma, we have 
$\hat\s((T_g\pr)^*\alpha_{[g]})=\hat\tg((T_h\pr)^*\alpha_{[h]})$
and $(T_g\pr)^*\alpha_{[g]}\star (T_h\pr)^*\alpha_{[h]}
=(T_{g\star h}\pr)^*(\alpha_{[g]}\star \alpha_{[h]})$. 
The pairs
$(v_g-w_g, (T_g\pr)^*\alpha_{[g]})\in\mathsf D_G(g)$ and 
$(v_h, (T_h\pr)^*\alpha_{[h]})\in\mathsf D_G(h)$
are hence compatible and their product, 
$$((v_g -w_g)\star v_h, (T_{g\star h}\pr)^*(\alpha_{[g]}\star \alpha_{[h]}))$$
is an element of $\mathsf D_G(g\star h)$.
Since it pushes forward to $(v_{[g]}\star v_{[h]}, \alpha_{[g]}\star \alpha_{[h]})$, 
we find that $(v_{[g]}\star v_{[h]}, \alpha_{[g]}\star \alpha_{[h]})\in 
\pr(\mathsf D_G)([g]\star[h])$.

It remains to show that the inverse $(v_{[g]},\alpha_{[g]})\inv$
is an element of $\pr(\mathsf D_G)([g]\inv)$. Recall
that $[g]\inv=[g\inv]$. We have 
$(v_g,(T_g\pr)^* \alpha_{[g]})\inv\in\mathsf D_G(g\inv)$.
Since 
\[\left((T_g\pr)^* \alpha_{[g]}\right)\star \left((T_{g\inv}\pr)^*\alpha_{[g]}\inv\right)
=\hat\tg((T_g\pr)^* \alpha_{[g]})\]
and 
in the same manner
\[\left((T_{g\inv}\pr)^*\alpha_{[g]}\inv\right)\star \left((T_g\pr)^* \alpha_{[g]}\right)
=\hat\s((T_g\pr)^* \alpha_{[g]}),\] we find that 
$(T_{g\inv}\pr)^*\alpha_{[g]}\inv=((T_g\pr)^* \alpha_{[g]})\inv$.
Since $(\pr,\pr_\circ)$ is a morphism of Lie groupoids, we find also
that $T_{g\inv}\pr (v_g\inv)=(T_g\pr v_g)\inv=v_{[g]}\inv$.
Thus,\linebreak $(v_g,(T_g\pr)^* \alpha_{[g]})\inv\in\mathsf D_G(g\inv)$
pushes forward to $(v_{[g]},\alpha_{[g]})\inv$, which is consequently
an element of $\pr(\mathsf D_G)([g]\inv)$.
\end{proof}

\begin{example}[Trivial example]
Let $G\rr P$ be a $\tg$-connected Lie groupoid and 
 $S\subseteq TG$ a  multiplicative, involutive 
subbundle. Then, if $S^\circ \subseteq T^*G$ is the annihilator
of $S$, the  Dirac structure $\mathsf D=S\oplus S^\circ$
is multiplicative with characteristic distribution equal to $S$. 
If $S$ is complete and simple, then the reduced 
Poisson groupoid is just the trivial Poisson groupoid $(G/S\rr P/S, \pi=0)$.
\end{example}

\begin{remark}
In the situation of the previous theorem,
the multiplicative 
subbundle $\mathsf{G_0}$  of $TG$ has constant rank on $G$. In particular, 
the intersection $TP\cap\mathsf{G_0}$ is a smooth vector bundle over $P$
and for each $g\in G$, the restriction 
to $\mathsf{G_0}(g)$
of the target map, $T_g\tg:\mathsf{G_0}(g)\to \mathsf{G_0}(\tg(g))\cap T_{\tg(g)}P$,
is surjective (see Lemma \ref{constant_rank}). 
By a  theorem in \cite{Jotz11c}, there exists
then a Dirac structure $\mathsf D_P$ on $P$ such that 
$\tg: (G, \mathsf D_G)\to (P,\mathsf D_P)$
is a forward Dirac map.
Since $(G/\mathsf{G_0}\rr P/\mathsf{G_0}, \pr(\mathsf D_G))$
is a Poisson groupoid,  we know also
by a theorem in \cite{Weinstein88b} that there is   
a Poisson structure $\{\cdot\,,\cdot\}_{P/\mathsf{G_0}}$ on $P/\mathsf{G_0}$ such that
$[\tg]:(G/\mathsf{G_0}, \pr(\mathsf D_G))\to (P/\mathsf{G_0},\{\cdot\,,\cdot\}_{P/\mathsf{G_0}})$
is a forward Dirac map.
It is easy to check that
the map $\pr_\circ:(P,\mathsf D_P)\to (P/\mathsf{G_0},\{\cdot\,,\cdot\}_{P/\mathsf{G_0}})$
is then also a forward Dirac map, i.e.,
$\mathsf D_{\{\cdot\,,\cdot\}_{P/\mathsf{G_0}}}$ is the forward Dirac image 
of $\mathsf D_P$ under $\pr_\circ$.  
\end{remark}

\begin{remark}
In the Lie group case, the Poisson Lie group $(G/N, q(\mathsf{D}_G))$ associated to a
 Dirac Lie group
$(G, \mathsf D_G)$ satisfying the necessary regularity assumptions 
was also a \emph{Poisson homogeneous space}
of the Dirac Lie group. In general, the Poisson groupoid 
associated to the Dirac  groupoid is \emph{not} a Poisson homogeneous space of the Dirac groupoid
since the quotient $G/\mathsf{G_0}$ is not a homogeneous space of the Lie groupoid $G\rr P$.
\end{remark}

\section*{Acknowledgments}
 I am very grateful to my advisor Tudor Ratiu for many discussions, for
his enthusiasm and his good advice.
It is a pleasure to dedicate this note, which is a part of my Ph.D. thesis, to him.
I would also like to thank Marco Zambon for Example 
\ref{example_of_marco}, Henrique Bursztyn for Example \ref{example_of_henrique}
and the referees for their comments that have improved a lot the exposition.

\def\cprime{$'$} \def\polhk#1{\setbox0=\hbox{#1}{\ooalign{\hidewidth
  \lower1.5ex\hbox{`}\hidewidth\crcr\unhbox0}}}
\providecommand{\bysame}{\leavevmode\hbox to3em{\hrulefill}\thinspace}
\providecommand{\MR}{\relax\ifhmode\unskip\space\fi MR }
\providecommand{\MRhref}[2]{%
  \href{http://www.ams.org/mathscinet-getitem?mr=#1}{#2}
}
\providecommand{\href}[2]{#2}

\medskip
% The data information below will be filled by AIMS editorial staff
Received xxxx 20xx; revised xxxx 20xx.
\medskip


\begin{thebibliography}{MMndDMn06}

\bibitem[CDW87]{CoDaWe87}
A.~Coste, P.~Dazord, and A.~Weinstein, \emph{Groupo\"\i des symplectiques},
  Publications du {D}\'epartement de {M}ath\'ematiques. {N}ouvelle {S}\'erie.
  {A}, {V}ol.\ 2, Publ. D\'ep. Math. Nouvelle S\'er. A, vol.~87, Univ.
  Claude-Bernard, Lyon, 1987, pp.~i--ii, 1--62.

\bibitem[Cou90]{Courant90a}
T.~J. Courant, \emph{{Dirac manifolds.}}, Trans. Am. Math. Soc. \textbf{319}
  (1990), no.~2, 631--661.

\bibitem[HN91]{HiNe91}
J.~Hilgert and K.-H. Neeb, \emph{{Lie {G}roups and {L}ie {A}lgebras.
  ({L}ie-{G}ruppen und {L}ie-{A}lgebren.)}}, {Braunschweig: Vieweg. 361 S. },
  1991.

\bibitem[Ili06]{Iliev06}
B.~Z. Iliev, \emph{{Handbook of {N}ormal {F}rames and {C}oordinates.}},
  {Progress in Mathematical Physics 42. Basel: Birkh\"auser. xv, 441~p. },
  2006.

\bibitem[IMMdDM08]{IgMaMaMa08}
D.~Iglesias, J.~C. Marrero, D.~Mart{\'\i}n~de Diego, and E.~Mart{\'\i}nez,
  \emph{{Discrete nonholonomic Lagrangian systems on Lie groupoids.}}, J.
  Nonlinear Sci. \textbf{18} (2008), no.~3, 221--276 (English).

\bibitem[JO11]{JoOr11}
M.~Jotz and C.~Ortiz, \emph{Foliated groupoids and their infinitesimal data},
  Preprint, arXiv:1109.4515v1 (2011).

\bibitem[Jot10]{Jotz11c}
M.~Jotz, \emph{{Infinitesimal objects associated to {D}irac groupoids and their
  homogeneous spaces.}}, Preprint, arXiv:1009.0713 (2010).

\bibitem[Jot11a]{thesis}
\bysame, \emph{Dirac {G}roup(oid)s and {T}heir {H}omogeneous {S}paces}, Ph.D.
  thesis, EPFL, Lausanne, 2011.

\bibitem[Jot11b]{Jotz11a}
\bysame, \emph{{Dirac Lie groups, {D}irac homogeneous spaces and the {T}heorem
  of {D}rinfel'd.}}, arXiv:0910.1538, to appear in ``Indiana University
  Mathematics Journal'' (2011).

\bibitem[JR{\'S}11]{JoRaSn11}
M.~Jotz, T.S. Ratiu, and J.~{\'S}niatycki, \emph{Singular {D}irac reduction},
  Trans. Amer. Math. Soc. \textbf{363} (2011), 2967--3013.

\bibitem[JRZ11]{JoRaZa11}
M.~Jotz, T.~Ratiu, and M.~Zambon, \emph{Invariant frames for vector bundles and
  applications}, Geometriae Dedicata (2011), 1--12.

\bibitem[Mac87]{Mackenzie87}
K.C.H. Mackenzie, \emph{Lie {G}roupoids and {L}ie {A}lgebroids in
  {D}ifferential {G}eometry}, London Mathematical Society Lecture Note Series,
  vol. 124, Cambridge University Press, Cambridge, 1987.

\bibitem[Mac00]{Mackenzie00}
\bysame, \emph{Double {L}ie algebroids and second-order geometry. {II}}, Adv.
  Math. \textbf{154} (2000), no.~1, 46--75.

\bibitem[Mac05]{Mackenzie05}
\bysame, \emph{General {T}heory of {L}ie {G}roupoids and {L}ie {A}lgebroids},
  London Mathematical Society Lecture Note Series, vol. 213, Cambridge
  University Press, Cambridge, 2005.

\bibitem[MM03]{MoMr03}
I.~Moerdijk and J.~Mr\v{c}un, \emph{{Introduction to {F}oliations and {L}ie
  {G}roupoids.}}, {Cambridge Studies in Advanced Mathematics. 91. Cambridge:
  Cambridge University Press. ix, 173 p. }, 2003.

\bibitem[MMndDMn06]{MaMaMa06}
J.~C. Marrero, D.~Mart\'\i n~de Diego, and E.~Mart\'\i~nez, \emph{{Discrete
  Lagrangian and Hamiltonian mechanics on Lie groupoids.}}, Nonlinearity
  \textbf{19} (2006), no.~6, 1313--1348 (English).

\bibitem[OR04]{OrRa04}
J.-P. Ortega and T.S. Ratiu, \emph{{Momentum {M}aps and {H}amiltonian
  {R}eduction.}}, {Progress in Mathematics (Boston, Mass.) 222. Boston, MA:
  Birkh\"auser. xxxiv, 497~p. }, 2004.

\bibitem[Ort08]{Ortiz08}
C.~Ortiz, \emph{{Multiplicative Dirac structures on Lie groups.}}, C. R.,
  Math., Acad. Sci. Paris \textbf{346} (2008), no.~23-24, 1279--1282.

\bibitem[Ort09]{Ortiz08t}
\bysame, \emph{Multiplicative {D}irac {S}tructures}, Ph.D. thesis, Instituto de
  Matem\'atica Pura e Aplicada, 2009.

\bibitem[Pra88]{Pradines88}
J.~Pradines, \emph{Remarque sur le groupo\"\i de cotangent de
  {W}einstein-{D}azord}, C. R. Acad. Sci. Paris S\'er. I Math. \textbf{306}
  (1988), no.~13, 557--560.

\bibitem[Ste74]{Stefan74a}
P.~Stefan, \emph{Accessible sets, orbits, and foliations with singularities},
  Proc. London Math. Soc. (3) \textbf{29} (1974), 699--713.

\bibitem[Ste80]{Stefan80}
\bysame, \emph{Integrability of systems of vector fields}, J. London Math. Soc.
  (2) \textbf{21} (1980), no.~3, 544--556.

\bibitem[Sus73]{Sussmann73}
H.J. Sussmann, \emph{Orbits of families of vector fields and integrability of
  distributions}, Trans. Amer. Math. Soc. \textbf{180} (1973), 171--188.

\bibitem[Wei88]{Weinstein88b}
A.~Weinstein, \emph{{Coisotropic calculus and Poisson groupoids.}}, J. Math.
  Soc. Japan \textbf{40} (1988), no.~4, 705--727.

\bibitem[Wei96]{Weinstein96}
\bysame, \emph{Lagrangian mechanics and groupoids}, Mechanics day ({W}aterloo,
  {ON}, 1992), Fields Inst. Commun., vol.~7, Amer. Math. Soc., Providence, RI,
  1996, pp.~207--231.

\bibitem[Zam08]{Zambon08}
M.~Zambon, \emph{Reduction of branes in generalized complex geometry}, J.
  Symplectic Geom. \textbf{6} (2008), no.~4, 353--378.

\bibitem[Zam10]{Zambon10}
\bysame, \emph{Submanifolds in poisson geometry: a survey}, Complex and
  {D}ifferential {G}eometry, Springer {P}roceedings in {M}athematics, vol.~8,
  Springer Berlin, 2010, pp.~403--420.

\end{thebibliography}
\end{document}